\documentclass{amsart}

\usepackage{xparse}
\usepackage{etoolbox}
\usepackage{kvoptions}

\usepackage{fix-cm}
\usepackage[T1]{fontenc}
\usepackage[utf8]{inputenc}
\usepackage[full]{textcomp}
\usepackage[l2tabu, orthodox]{nag}

\usepackage{lmodern} 
\usepackage[bitstream-charter]{mathdesign}

\usepackage{microtype}
\usepackage{csquotes}

\usepackage{graphicx}
\usepackage{tikz}
\usetikzlibrary{matrix, arrows, decorations.markings, decorations.pathmorphing, calc}

\tikzset{->-/.style={decoration={
  markings,
  mark=at position 0.5 with {\arrow{stealth}}},postaction={decorate}}}
\tikzset{->>-/.style={decoration={
  markings,
  mark=at position 0.5 with {\arrow{>>}}},postaction={decorate}}}
\tikzset{snake it/.style={decorate, decoration=snake}}

\usepackage{tikz-cd}
\usepackage{tikz-3dplot}
\usepackage{wrapfig}

\usepackage{caption}
\usepackage{subcaption}
\usepackage{multirow}
\usepackage{booktabs}
\usepackage{tabularx}
\usepackage{tabulary}
\usepackage{longtable}
\usepackage{siunitx}

\usepackage{dsfont}
\usepackage{xfrac}
\usepackage{mathtools}
\usepackage{mleftright}
\usepackage{manfnt}
\usepackage{accents}
\usepackage{faktor}

\DeclareFontFamily{U}{rsfs}{\skewchar\font127 }
\DeclareFontShape{U}{rsfs}{m}{n}{%
   <-6> rsfs5
   <6-8> rsfs7
   <8-> rsfs10
}{}

\usepackage{emptypage}
\usepackage{comment}

\usepackage[inline]{enumitem}
\usepackage[obeyFinal,colorinlistoftodos=true,textsize=footnotesize]{todonotes}

\usepackage[hyphens]{url}

\usepackage{hyperref}
\usepackage{geometry} 
\usepackage{amsthm}
\usepackage{thmtools,thm-restate}
\usepackage[capitalize]{cleveref}

\definecolor{dark-red}{rgb}{0.4,0.15,0.15}
\definecolor{dark-blue}{rgb}{0.15,0.15,0.4}
\definecolor{medium-blue}{rgb}{0,0,0.5}
\hypersetup{
    colorlinks, linkcolor={dark-red},
    citecolor={dark-blue}, urlcolor={medium-blue}
}

\LetLtxMacro{\amsmathdots}{\dots}
\usepackage{ellipsis}
\AtBeginDocument{\LetLtxMacro{\dots}{\amsmathdots}}

\DeclareMathOperator{\tr}{Tr}

\DeclareMathOperator{\Stab}{Stab}

\DeclareMathOperator{\GL}{GL}
\DeclareMathOperator{\PGL}{PGL}

\DeclareMathOperator{\Mat}{Mat}

\DeclareMathOperator{\PConf}{PConf}
\DeclareMathOperator{\UConf}{UConf}
\DeclareMathOperator{\Frob}{Frob}

\newcommand*{\et}{\text{ét}}

\newcommand*{\dispunct}[1]{\,\text{#1}}

\newcommand{\from}{\vcentcolon}

\newcommand{\dsum}{\oplus}
\newcommand{\dSum}{\bigoplus}
\newcommand{\tensor}{\otimes}

\NewDocumentCommand\xDeclarePairedDelimiter{mmm}
 {%
  \NewDocumentCommand#1{som}{%
   \IfNoValueTF{##2}
    {\IfBooleanTF{##1}{#2##3#3}{\mleft#2##3\mright#3}}
    {\mathopen{##2#2}##3\mathclose{##2#3}}%
  }%
 }

\xDeclarePairedDelimiter{\abs}{\lvert}{\rvert}
\xDeclarePairedDelimiter{\norm}{\lVert}{\rVert}
\xDeclarePairedDelimiter{\floor}{\lfloor}{\rfloor}
\xDeclarePairedDelimiter{\ceil}{\lceil}{\rceil}
\xDeclarePairedDelimiter{\gen}{\langle}{\rangle}
\xDeclarePairedDelimiter{\pseries}{\llbracket}{\rrbracket}

\NewDocumentCommand{\set}{somm}{%
   \IfNoValueTF{#2}
    {\IfBooleanTF{#1}{\{#3 \mid #4\}}{\mleft\{ #3 \mathrel{}\middle\vert\mathrel{} #4 \mright\}}}
    {\mathopen{#2\{}#3 \mathrel{}#2\vert\mathrel{} #4\mathclose{#2\}}}%
  }
\NewDocumentCommand{\present}{somm}{%
   \IfNoValueTF{#2}
    {\IfBooleanTF{#1}{\langle#3 \mid #4\rangle}{\mleft\langle#3 \mathrel{}\middle\vert\mathrel{} #4 \mright\rangle}}
    {\mathopen{#2\langle}#3 \mathrel{}#2\vert\mathrel{} #4\mathclose{#2\rangle}}%
  }
\NewDocumentCommand{\inner}{somm}{%
   \IfNoValueTF{#2}
    {\IfBooleanTF{#1}{\langle#3 , #4\rangle}{\mleft\langle#3 , #4 \mright\rangle}}
    {\mathopen{#2\langle}#3 , #4\mathclose{#2\rangle}}%
  }

\newcommand{\CC}{\mathbb{C}}
\newcommand{\FF}{\mathbb{F}}

\newcommand{\PP}{\mathbb{P}}
\newcommand{\QQ}{\mathbb{Q}}
\newcommand{\RR}{\mathbb{R}}

\newcommand{\cH}{\mathcal{H}}

\newcommand{\cU}{\mathcal{U}}

\newcommand{\fS}{\mathfrak{S}}

\newcolumntype{C}{>{\raggedright\arraybackslash}X}

\newcommand*{\widebar}[1]{\mkern 1.5mu\overline{\mkern-1.5mu#1\mkern-1.5mu}\mkern 1.5mu}

\newcommand*{\cl}[1]{
\begingroup
    \setbox\z@=\hbox{\ensuremath{#1}}%
    \ifdimgreater{\wd\z@}{4em}{\mleft(#1\mright)^{-}}{\widebar{#1}}
\endgroup
}

\newcommand*{\interior}[1]{
\begingroup
    \setbox\z@=\hbox{\ensuremath{#1}}%
    \ifdimgreater{\wd\z@}{1.5em}{\mleft(#1\mright)^{\circ}}{\accentset{\circ}{#1}}
\endgroup
}

\newcommand\isom{\xrightarrow{\,\smash{\raisebox{-0.6ex}{\ensuremath{\sim}}}\,}}

\numberwithin{equation}{section}
\declaretheorem[sibling=equation]{theorem}
\declaretheorem[sibling=theorem,style=remark]{example}
\declaretheorem[sibling=theorem,style=definition]{definition}
\declaretheorem[sibling=theorem]{lemma}

\declaretheorem[sibling=theorem]{corollary}
\declaretheorem[sibling=theorem]{proposition}

\declaretheorem[sibling=theorem, style=remark]{remark}
\declaretheorem[numbered=no, title=Theorem]{theorem*}
\declaretheorem[numbered=no, title=Corollary]{corollary*}

\usepackage{sseq}

\usepackage[style=ieee-alphabetic,url=false]{biblatex}

\renewbibmacro*{doi+eprint+url}{%
  \iftoggle{bbx:doi}
    {\iffieldundef{url}{\printfield{doi}}{}}
    {}%
  \newunit\newblock
  \iftoggle{bbx:eprint}
    {\usebibmacro{eprint}}
    {}%
  \newunit\newblock
  \iftoggle{bbx:url}
    {\usebibmacro{url+urldate}}
    {}}

\newcommand*{\cV}{\mathcal{V}}

\newlist{singularity}{enumerate}{2}
\setlist[singularity,1]{label=(\Roman*),noitemsep, ref=\Roman*}
\setlist[singularity,2]{label=(\alph*),noitemsep, ref=\alph*}

\newcommand*{\type}[1]{\text{\ref{#1}}}

\makeatletter
\def\paragraph{\@startsection{paragraph}{4}%
  \z@\z@{-\fontdimen2\font}%
  {\normalfont\bfseries}}
\makeatother

\begin{document}

\title{Cohomology of the universal smooth cubic surface}
\author{Ronno Das}
\address{Department of Mathematics, University of Chicago, Chicago, IL 60637, USA}
\email{ronno@math.uchicago.edu}

\begin{abstract}
We compute the rational cohomology of the universal family of smooth cubic surfaces using Vassiliev's method of simplicial resolution.
Modulo embedding, the universal family has cohomology isomorphic to that of $\PP^2$.
A consequence of our theorem is that over the finite field $\FF_q$, away from finitely many characteristics, the average number of points on a smooth cubic surface is $q^2 + q + 1$.
\end{abstract}

\maketitle

\section{Introduction}

A \emph{cubic surface} $S \subset \PP^3 = \CC \PP^3$ is the zero set $S = \cV(F)$ of a homogeneous polynomial $F$ of degree $3$ in $4$ variables.
The surface $S$ is singular (i.e.\ not smooth) if and only if the $20$ coefficients of $F$ are a zero of a discriminant polynomial $\Delta \from \CC^{20} \to \CC$.
Thus the space of smooth cubic surfaces is an open locus $M = M_{3,3} \vcentcolon = \PP^{19} \setminus \cV(\Delta)$.

The \emph{incidence variety}
\[U = \set{(S,p)}{p \in S} \subset M \times \PP^3\]
of points and cubic surfaces (see \cref{main-definition}) is a subvariety $U \subset M \times \PP^3$ and the canonical projection map $U \to M$ is a fiber bundle, whose fiber over $S \in M$ is exactly $S \subset \PP^3$.
This is the `universal family' of cubic surfaces with embeddings in $\PP^3$, in the sense that a family of embedded smooth cubic surfaces corresponds to a pullback of this bundle by a map to $M$.

The automorphism group of $\PP^3$ is $\PGL(4, \CC)$ and this group takes cubic surfaces to cubic surfaces, preserving smoothness.
In particular the projection map $\pi \from U \to M$ is $\PGL(4, \CC)$-equivariant.
Vassiliev showed (in \cite{Vassiliev99}) that the space $M$ has the same rational cohomology as $\PGL(4,\CC)$ and it following results of Peters--Steenbrink (\cite{PS03}) this isomorphism is induced by the orbit map given by $g \mapsto g(S_0)$, for any choice of $S_0 \in M$ (see~\cref{base-cohomology}).
See also Tommasi (\cite{Tommasi14}).

The main result of this paper is that the rational cohomology of $U$ is isomorphic to that of $\PGL(4,\CC) \times \CC \PP^2$.
\begin{theorem}[Cohomology of the universal smooth cubic]\label{main-theorem}
Let $\eta \in H^2(\CC \PP^3; \QQ)$ be the hyperplane class.
Let $\iota \from U \to M \times \CC \PP^3$ be the inclusion map.
Then $\iota^*(1 \tensor \eta^3) = 0$ and the induced map
\[H^*(M \times \CC \PP^3; \QQ)/(\eta^3) \to H^*(U; \QQ)\]
is an isomorphism.
In particular, with rational coefficients,
\[H^*(U) \cong H^*(M \times \CC \PP^2) \cong H^*(\PGL(4,\CC)) \tensor H^*(\CC \PP^2) \cong \QQ[\alpha_3, \alpha_5, \alpha_7, \eta]/(\alpha_3^2, \alpha_5^2, \alpha_7^2, \eta^3)\]
where $\alpha_i \in H^i(\PGL(4,\CC); \QQ)$.
Since the inclusion map is algebraic, each isomorphism is an isomorphism of mixed Hodge structures.
In particular, $H^k(U; \QQ)$ is pure of Tate type; each generator $\alpha_{2k-1}$ is of bidegree $(k,k)$ and $\eta$ is of bidegree $(1,1)$.
\end{theorem}

The key tool in our proof of \cref{main-theorem} is simplicial resolution à la Vassiliev.
Considering the combinatorics of how the marked point is situated with respect to possible singularities on the surfaces makes the casework fairly complicated.
We devote all of \cref{case-work} to this computation, while \cref{proof-section} contains the rest of the proof.

\subsection{Applications: moduli space, representations of \texorpdfstring{$W(E_6)$}{W(E\_6)} and point counts}\label{applications}

We now give a few applications of \cref{main-theorem}.

\subsubsection*{Cohomology of moduli spaces}
The map $\pi \from U \to M$ is $\PGL(4,\CC)$-equivariant and each orbit (in either $M$ or $U$) is closed (see~e.g.~\cite{ACT02}).
Further, two cubic surfaces are isomorphic exactly when they are in the same $\PGL(4,\CC)$-orbit.
Thus, passing to the geometric quotient gives a bundle
\[\cU_{3,3} \to \cH_{3,3} \dispunct,\]
where
\[\cH_{3,3} \vcentcolon= M/\PGL(4, \CC)\]
is the moduli space of smooth cubic surfaces and
\[\cU_{3,3} \vcentcolon= U/\PGL(4, \CC)\]
is the moduli space of cubic surfaces equipped with a point.
The induced map $\cU_{3,3} \to \cH_{3,3}$ is the universal family of cubic surface up to isomorphism.

Note that both $\cH_{3,3}$ and $\cU_{3,3}$ are coarse moduli spaces.
For example the Fermat cubic defined by $x^3 + y^3 + z^3 + w^3$ equipped with the point $[1:-1:0:0]$ has non-trivial (but finite) stabilizer in $\PGL(4, \CC)$.
Using a theorem of Peters and Steenbrink (\cite[Theorem 2]{PS03}), which is a generalization of the Leray--Hirsch theorem, we have the following corollary.
\begin{corollary}\label{moduli-space}
The space $\cU_{3,3}$ has the rational cohomology of $\PP^2$:
\[H^i(\cU_{3,3}; \QQ) \cong \begin{dcases*}
\QQ & if $i = 0, 2$ or $4$;\\
0 & otherwise.
\end{dcases*}
\]
\end{corollary}
For comparison, it was known previously that $\cH_{3,3}$ is $\QQ$-acyclic; see \cref{base-cohomology} below.

\subsubsection*{Monodromy and the normal cover with deck group \texorpdfstring{$W(E_6)$}{W(E\_6)}}
One way of trying to compute $H^*(U; \QQ)$ would be to use the fiber bundle $U \to M$.
Since the fiber over a surface $S \in M$ is exactly $S \subset \CC \PP^3$, this provides a spectral sequence
\[H^p(M; H^q(S)) \implies H^{p+q}(U) \dispunct,\]
where the coefficients are twisted by the monodromy action of $\pi_1(M)$ on
\[H^*(S; \QQ) = \begin{cases*}
\QQ & if $* = 0, 4$;\\
\QQ^7 & if $* = 2$;\\
0 & otherwise.
\end{cases*}\]
The monodromy action on $H^0$ and $H^4$ are of course trivial but the action on $H^2$ is quite interesting; to explore this we need a better description of $H^2(S)$.

There are different elements of $H^2(S)$ that can be described as the hyperplane class: the pullback $\eta$ of a generic hyperplane in $\PP^3$, which also equals the anticanonical class; or the strict transform $\lambda$ of a line when $S$ is identified with $\PP^2$ blown up at $6$ points.
Every cubic surface $S$ famously contains $27$ lines and a choice of any $6$ \emph{disjoint} lines out of the $27$ when blown down produces $\PP^2$ (see for instance \cite[Section V.4, specifically Proposition V.4.10]{Hartshorne77}).
It is then straightforward to see that the classes of $6$ such (disjoint) lines, along with either $\eta$ or $\lambda$ is a basis of $H^2(S)$.

The monodromy action keeps $\eta$ invariant since it preserves the embedding $S \subset \PP^3$, but it does not preserve the choice of lines---in fact it must be transitive on the choices of $6$ disjoint lines.
It does act by a finite group, the automorphism group of the intersection pairing of the $27$ lines, which can be identified as the Weyl group $W(E_6)$ of the root system $E_6$ (see \cite[Remark 23.8.2]{Manin86}, also \cite{Jordan89,Harris79}).
As a representation of $W(E_6)$, we get a decomposition of $H^2(S)$ into a one-dimensional trivial representation spanned by $\eta$ and a copy of the irreducible \emph{fundamental representation} of $W(E_6)$, denoted $V_{\text{fund}}$, spanned by the projections of any $6$ disjoint lines.
Thus,
\[H^p(M; H^2(S)) \cong H^p(M; \QQ\gen*{\eta}) \dsum H^p(M; V_{\text{fund}})\dispunct.\]

So to use the Serre spectral sequence, for $U \to M$, we would need to compute $H^p(M; V_{\text{fund}})$.
The finite quotient $\pi_1(M) \to W(E_6)$ corresponds to a normal cover $M(27)$ of $M$, whose points are given by decorating each $S \in M$ with a choice of ordering of the $27$ lines, consistent with some chosen intersection pattern.
Thus by transfer, we would need the multiplicity of $V_{\text{fund}}$ in $H^*(M(27); \QQ)$.
As the following corollary shows, it is in fact possible to turn this argument backwards and use \cref{main-theorem} to compute this multiplicity.

\begin{corollary}\label{fundamental-representation}
The fundamental representation $V_{\textup{fund}}$ of $W(E_6)$ does not appear in $H^*(M(27); \QQ)$.
\end{corollary}

\begin{remark}
Some of the other irreducible representations of $W(E_6)$ are also precluded from occurring in $H^*(M(27))$, see \cite[Corollary 1.4]{Das18}, but this is not sufficient to determine $H^*(M(27))$ entirely.
\end{remark}

\begin{proof}[Proof of \cref{fundamental-representation}]
By Bezout's theorem, $\eta^2 = \eta\cup \eta \in H^4(S)$ is $3$ times the fundamental cohomology class of $S$ and of course $\eta^3 = 0$.
Moreover the pullback of a generic hyperplane to $U$ under the map $U \to \PP^3$ further pulls back to $\eta$ for every inclusion $S \subset U$, so we also denote this class by $\eta \in H^2(U)$.
By \cref{main-theorem}, $H^*(U) = H^*(M) \tensor \QQ[\eta]$.

But in the Serre spectral sequence for the bundle $U \to M$ from above, the $E_2$ page has three rows ($q = 0, 2, 4$), which consist of $H^p(M) \tensor \QQ\gen*{\eta^{q/2}}$ along with $H^p(M; V_{\text{fund}})$ on the $q = 2$ row.
There cannot be a non-zero differential mapping either into or out of $H^p(M) \tensor \QQ[\eta]$, since these terms must survive till the $E_\infty$ page and thus all the differentials vanish.
But $H^*(M) \tensor \QQ[\eta]$ already accounts for all of  $H^*(U)$, so we must have
\[H^p(M; V_{\text{fund}}) = 0\]
for each $p$.
But by transfer,
\[H^p(M; V_{\text{fund}}) \cong H^p(M(27)) \tensor_{W(E_6)} V_{\text{fund}} \dispunct,\]
so this irreducible representation cannot occur in any $H^p(M(27))$.
\end{proof}

\begin{remark}
The vanishing of the differentials is consistent with the bundle $U \to M$ having a (continuous) section.
In fact, the existence of such a section, along with the result that $H^*(M; V_{\text{fund}}) = 0$ would be sufficient to recover \cref{main-theorem}.
\end{remark}

\subsubsection*{Point counts over $\FF_q$}
The spaces $U$ and $M$ as defined are (the complex points of) quasiprojective varieties defined by integer polynomials.
To be more explicit, the discriminant $\Delta$ is an integer polynomial, as are the polynomials defining the incidence of a point and a cubic surface.
For a finite field $\FF_q$ of characteristic $p$, we can base change to $\FF_q$.
That is, reducing the defining polynomials $\operatorname{mod} p$ defines spaces
\[M(\FF_q) \subset \PP^{19}(\FF_q) \dispunct, \qquad \text{ and } \qquad U(\FF_q) \subset \PP^{19}(\FF_q) \times \PP^3(\FF_q) \dispunct,\]
and a projection map
\[\pi \from U(\FF_q) \to M(\FF_q) \dispunct.\]

For $p \ne 3$, the discriminant $\Delta$ continues to characterize singular polynomials, so $M(\FF_q)$ is the space of smooth cubic surfaces defined over $\FF_q$ (where a homogeneous cubic polynomial is smooth if it is smooth at all $\widebar{\FF}_q$ points).
Similarly, $U(\FF_q)$ is the space of pairs $(S,p)$ of smooth cubic surfaces $S$ and points $p$ defined over $\FF_q$ such that $p \in S$.
Thus, $\dfrac{\# U(\FF_q)}{\# M(\FF_q)}$ is the average number of $\FF_q$ points on a cubic surface defined over $\FF_q$.

For a smooth quasiprojective variety $Y$, the $\FF_q$ points are exactly the fixed points of $\Frob_q$ on $Y(\widebar{\FF}_q)$ and $\#Y(\FF_q)$ is determined by the Grothendieck--Lefschetz fixed point formula (see~e.g.~\cite{Milne13}):
\begin{equation}\label{GL-formula}
\# Y(\FF_q) = q^{\dim Y}\sum_{i \ge 0} (-1)^i \tr(\Frob_q \from H^i_{\et}(Y; \QQ_\ell)^\vee) \dispunct,
\end{equation}
where $\ell$ is a prime other than $p$.
Further, there are comparison theorems implying isomorphisms
\[H^i_\et(Y; \QQ_\ell) \cong H^i(Y(\CC); \QQ_\ell) \cong H^i(Y(\CC); \QQ) \tensor \QQ_\ell \dispunct,\]
away from a finite set of characteristics (see~e.g.~\cite[Théorème 1.4.6.3, Théorème 7.1.9]{Deligne77}).
This formula lets us use our results to deduce consequences about $\#U(\FF_q)$.

Applying the fixed-point formula to each $S$ we get $\# S(\FF_q) = q^2 + (t+1)q + 1$, where $t$ is the trace of Frobenius on the complement of $\eta \in H^2(S)$ described above.
Frobenius must act on $H^2(S)$ by some element of $W(E_6)$, so the possible values of $t$ are given by the character of $W(E_6)$ on the fundamental representation, namely the set $\{-3,-2,-1,0,1,2,3,4,6\}$.
Serre asked (e.g. in \cite[Section 2.3.3]{Serre12}) which $t$ can occur over all surfaces $S$ defined over each $q$.

\begin{theorem}[Swinnerton-Dyer (\cite{SwinnertonDyer10}), Banwait--Fité--Loughran (\cite{BFL16})]
For $q = 2$, $3$ or $5$, the value $t = 6$ is impossible.
These are the only exceptions.
That is, for every other possible value of $t$ and $q$, there is some cubic surface over $\FF_q$ with that value of $t$.
\end{theorem}

Fixing a $q$, the average number of over all $S$ has to be $q^2 + q + 1 + q (t_{\text{average}})$.
The following corollary of \cref{main-theorem} shows that $t_{\text{average}} = 0$.

\begin{corollary}\label{average-points}
There is a finite set of characteristics, so that for a fixed $q = p^d$ with $p$ not in this set,
\[\# U(\FF_q) = q^{10}(q^2+q+1)(q^4-1)(q^3-1)(q^2-1)\dispunct.\]
Thus the average number of points defined over $\FF_q$ on a smooth cubic surface defined over $\FF_q$ is exactly $\#\PP^2(\FF_q) = q^2 + q + 1$.
\end{corollary}

To the best of our knowledge, the point count for $U(\FF_q)$ and the consequence about the average number of points is new.
The average number of points on \emph{irreducible} (but not necessarily smooth) cubic surfaces was known to also be $q^2+q+1$ by N.~Elkies (see \cite[Section 2.4]{Kaplan13}) using different methods.

\begin{proof}[Proof of \cref{average-points}]
By the Grothendieck--Lefschetz fixed point formula (\cref{GL-formula}) and \cref{main-theorem} we obtain
\[\# U(\FF_q) = \# (M \times \PP^2)(\FF_q) = q^4 \cdot (\#\PGL(4,\FF_q)) \cdot (\#\PP^2(\FF_q)) \dispunct.\qedhere\]
\end{proof}

\begin{remark}
For comparison, by \cref{base-cohomology},
\[\#M(\FF_q) = q^4 \cdot (\#\PGL(4,\FF_q)) \dispunct.\]
The $q^4$ factor arises from the difference in dimensions of $M$ and $\PGL(4)$.
\end{remark}

\subsection{Acknowledgments}
I thank Benson Farb for suggesting the problem and for his invaluable advice and comments throughout the composition of this paper.
I am also grateful to Weiyan Chen, Nir Gadish, Sean Howe, Akhil Mathew and Nathaniel Mayer for many helpful conversations.

\section{Rational cohomology of the incidence variety}\label{main-section}

\subsection{Definitions and setup}
Much of the following is analogous to \cite{Das18}, although here we are looking at the incidence variety of smooth cubic surfaces and points instead of \emph{lines}.
From now on we will work over the field $\CC$ of complex numbers.

Let $X = X_{3,3}$ be the space of \emph{smooth} homogeneous degree $3$ (complex) polynomials in $4$ variables, seen as a subset of $\CC[x,y,z,w]_3 \cong \CC^{20}$.
It will be important for us to note that smoothness of such a polynomial is defined by a `discriminant' $\Delta$: there is a homogeneous polynomial $\Delta \from \CC^{20} \to \CC$ with integer coefficients so that a polynomial $F \in \CC[x,y,z,w]_3$ is \emph{not smooth} if and only if $\Delta(F) = 0$.
Denoting the \emph{discriminant locus} by $\Sigma = \cV(\Delta) \subset \CC^{20}$,
\[X = \CC^{20} \setminus \Sigma \subset \CC^{20} \setminus \{0\} \dispunct.\]

Two polynomials $F$ and $F'$ in $\CC[x,y,z,w]_3$ define the same cubic surface ($\cV(F) = \cV(F')$) exactly when they are scalar multiples, that is, $F' = \lambda F$ for some $\lambda \in \CC^\times$.
Further, $F$ is smooth if and only if $\lambda F$ is smooth.
Thus we can quotient by $\CC^\times$ and get the space
\[M = X/\CC^\times \subset \PP^{19}\]
of smooth cubic surfaces.

Next we have the `incidence variety' of cubic polynomials and points
\[\Pi = \set{(F,p)}{F(p) = 0} \subset \CC[x,y,z,w]_3 \times \PP^3 \dispunct.\]
The preimage of $X$ under the projection $\pi \from (F,p) \mapsto F$ is the incidence variety of \emph{smooth} polynomials and points and will be denoted by $\widetilde X$.
Again taking the quotient by $\CC^\times$, we get the incidence variety of smooth cubic surfaces and points:
\begin{equation}\label{main-definition}
U = \widetilde X/\CC^\times = \set{(S,p) \in M \times \PP^3}{ p \in S} \subset M \times \PP^3\dispunct.
\end{equation}
The projection $U \to M$ is a fiber bundle, which we will also denote by $\pi$.

Each of the incidence varieties also comes equipped with another projection, to $\PP^3$; each of these maps is in fact a fiber bundle ($\Pi \to \PP^3$ happens to be a vector bundle).
We will denote the fiber over $p \in \PP^3$ in $\Pi$, $\widetilde X$ and $U$ by $\Pi_p \cong \CC^{19}$, $X_p$ and $U_p$ respectively.
To be explicit, $\Pi_p$ is the space of (not necessarily smooth) cubic polynomials that vanish at $p$, $X_p$ is the subset of \emph{smooth} cubic polynomials that vanish at $p$ and $U_p$ is the space of smooth cubic \emph{surfaces} that contain $p$.

All the spaces and maps above far fit into the following commuting diagram:
\begin{equation}\label{bundle-map-on-projective-space}
\begin{tikzcd}
X_p \arrow[rd, hook] \arrow[rrr, "\CC^\times"] & & & U_p \arrow[rd, hook] &  \\
 & \widetilde X \arrow[ld, "\pi"] \arrow[rrr, "\CC^\times"] & & & U \arrow[ld, "\pi"] \\
X \arrow[rrr, near end, "\CC^\times"] & & & M &  \\
 & \PP^3 \arrow[rrr, equal] \arrow[from=uu, crossing over] & & & \PP^3 \arrow[from=uu, crossing over]
\end{tikzcd}
\end{equation}

The actions of $\GL(4) \vcentcolon = \GL(4,\CC)$ on $\CC^4$ and $\PGL(4) = \GL(4)/(\CC^\times I)$ on $\PP^3$ induce actions on the spaces defined above: on $\Pi$, $X$ and $\widetilde X$ by $\GL(4)$; on $M$ and $U$ by $\PGL(4)$.
Fixing a point $p \in \PP^3$, the respective stabilizers in $\GL(4)$ and $\PGL(4)$ act on the fibers $X_p$ and $U_p$.
Choose a basepoint $(F_0,p_0) \in \widetilde X$ and set $S_0 = \cV(F_0)$ so that $(S_0,p_0) \in U$.
Then the actions produce orbit maps $g \mapsto g(S_0,p_0) = (g\cdot S_0, g \cdot p_0)$ and so on.
Since all the actions are compatible by construction, we also have the following commuting `cube':
\begin{equation}\label{bundle-map-orbit}
\begin{tikzcd}
\Stab_{\GL(4)}(p) \arrow[rrd, hook, dotted] \arrow[r] \arrow[d] & X_p \arrow[rrd, hook, dotted] \arrow[d] \\
\Stab_{\PGL(4)}(p) \arrow[rrd, hook, dotted] \arrow[r] & U_p \arrow[rrd, hook, dotted] & \GL(4) \arrow[r] \arrow[d]& \widetilde X \arrow[d]\\
& & \PGL(4) \arrow[r] & U\\
\end{tikzcd}
\end{equation}
All the horizontal maps are orbit maps, all the vertical maps are quotients by $\CC^\times$ and the diagonal dotted maps are inclusions of fibers over $p \in \PP^3$.

\begin{remark}
Since $\widetilde X$ and $X_p$ are connected, a different choice of basepoint $(F_0, p_0) \in \widetilde X$ does not change the orbit maps up to homotopy.
\end{remark}

\begin{theorem}[Vassiliev \cite{Vassiliev99}, Peters--Steenbrink \cite{PS03}] \label{base-cohomology}
The map $\PGL(4) \to M$ given by $g \mapsto g(S_0)$ induces an isomorphism
\[H^*(M; \QQ) \isom H^*(\PGL(4); \QQ) \dispunct.\]
\end{theorem}

\subsection{Proof of \texorpdfstring{\cref*{main-theorem}}{Theorem 1.1} and the role of simplicial resolution}
\label{proof-section}
Vassiliev's computation of $H^*(M; \QQ)$ and $H^*(X; \QQ)$ (as in \cite{Vassiliev99}) starts with a reduction, via Alexander duality, to computing the (Borel--Moore) homology of the discriminant locus $\Sigma$.
The space $\Sigma$, the set of singular cubic surfaces, is itself highly singular and stratifies based on the singular locus of an $F \in \Sigma$.
This stratification then produces a spectral sequence converging to $\widebar H_*(\Sigma) = H_*^{\mathrm{BM}}(\Sigma)$ (Borel--Moore or compactly supported homology).

Similar to \cite{Das18}, we apply the same methods to each fiber $X_p$, over $p$, of the map $\widetilde X \to \PP^3$, since each is a `discriminant complement' in the vector space $\Pi_p$ of polynomials vanishing at $p$.
We then need to stratify $\Sigma_p = \Sigma \cap \Pi_p$ not just by what the singular loci are as subsets of $\PP^3$, but also how they are configured with respect to the point $p$.
These are the types and subtypes described in \cref{plan}.
For now we will assume that we can perform this computation (which takes up all of \cref{case-work}) and when needed we refer to the answer described in \cref{cohomology-of-X-p}.

\begin{lemma}[{\cite[Lemma 2.7]{Das18}}]\label{hypersurface-complement}
Let $f \from \CC^n \to \CC$ be a non-constant homogeneous polynomial of degree $d$, so that $\cV(f)$ is a conical hypersurface; denote its complement by $Y = \CC^n \setminus \cV(f)$.
Let $\PP Y = Y /\CC^\times = \PP^{n-1} \setminus \cV_\PP(f)$ be the complement of the projective hypersurface given by the same polynomial $f$.
Then $H^*(Y; \QQ) \cong H^*(\CC^\times) \tensor H^*(\PP Y)$.
\end{lemma}

\begin{lemma}\label{surjection}
For a fixed $p \in \PP^3$, choose a complement hyperplane $H = \PP V$, with $V \subset \CC^4$.
Then $\GL(V) \subset \Stab_{\GL(4)}(p)$ acts on $X_p$.
For a choice of basepoint $F_0 \in X_p$, the orbit map $\GL(V) \to X_p$ given by $g \mapsto g(F_0) = F_0 \circ g$ induces a surjection
\[H^*(X_p; \QQ) \twoheadrightarrow H^*(\GL(V); \QQ) \cong H^*(\GL(3) ; \QQ) \dispunct.\]
\end{lemma}

\begin{proof}
Choose a basis of $V$ and denote the corresponding projective flag by $P \in L \subset H$.
This identifies $\GL(V)$ with $\GL(3, \CC)$.

As in the computation of $H^*(X_p; \QQ)$ in \cref{case-work}, it is important to identify, via Alexander duality, $H^*(X_p; \QQ)$ with $\widebar H_*(\Sigma_p)$ and similarly $H^*(\GL(3); \QQ)$ with $\widebar H_*(\Mat(3) \setminus \GL(3))$, where $\Mat(3)$ is the space of all $3 \times 3$ matrices.
The generators of $H^*(\GL(3); \QQ)$ (as a ring) are represented by the locus of matrices whose first $i$ columns are linearly dependent\footnote{For $i = 1$ this means the first column is $0$.
This description of the generators generalizes to $\GL(n) \subset M(n)$.}, for $i = 1, 2, 3$.

The orbit map extends to a map
\[\Mat(3) \to \Pi_p = X_p \cup \Sigma_p \dispunct.\]
It is enough to identify subspaces of $\Sigma_p$ that pull back to (a rational multiple of) the corresponding subspaces of $\Mat(2) \times \Mat(2)$.
Then by the proof of \cite[Lemma 7]{PS03} (which is the analogous statement for all singular polynomials, while $X_p$ restricts to polynomials vanishing at $p$), appropriate choices of subspaces are the sets of polynomials that are: (i) singular at $P$, (ii) singular at some point of $L$, (iii) singular at some point of $H$.
\end{proof}

\begin{remark}
The stabilizer of $p$ in $\PGL(4)$ deformation retracts to $\Stab(p,H) \cong \GL(V)$, for any choice of complement $H = \PP V$ above.
However, there isn't a way of extending the action of $\GL(V)$ on $X_p$ to an action of $\PGL(4)$ on $\widetilde X$.
\end{remark}

This allows us to apply Leray--Hirsch to $X_p \to X_p/\GL(V)$ by \cite[Theorem 2]{PS03}.
Knowing the Betti numbers of $X_p$ from \cref{cohomology-of-X-p} and using \cref{hypersurface-complement} to move from $X_p$ to $U_p$, we get the following.

\begin{corollary}\label{fiber-cohomology-ring}
As rings,
\[H^*(X_p; \QQ) \cong H^*(S^1 \times S^3 \times S^5 \times S^5; \QQ)\]
and
\[H^*(U_p; \QQ) \cong H^*(S^3 \times S^5 \times S^5; \QQ) \dispunct.\]
\end{corollary}

Now we can prove \cref{main-theorem}.
\begin{proof}[Proof of \cref{main-theorem}]\label{main-theorem-proof}
Let us supress rational coefficients for brevity.
Setting $G_p = \Stab_{\PGL(4)}(p) \simeq \GL(3)$, we have maps of bundles (as in \eqref{bundle-map-orbit}):
\[\begin{tikzcd}
G_p \arrow[rd, hook] \arrow[r] & U_p \arrow[rd, hook] \arrow[r] & M \arrow[rd, hook] \\
& \PGL(4) \arrow[r] \arrow[d] & U \arrow[r] \arrow[d] & M \times \PP^3 \arrow[d] \\
& \PP^3 \arrow[r,equal] & \PP^3 \arrow[r,equal] & \PP^3
\end{tikzcd}\]
The pair of horizontal maps on the left are orbit maps as described above and the pair on the right are inclusions.

Since the base is simply connected, we get spectral sequences
\[H^p(\PP^3) \tensor H^q(G_p) \implies H^{p+q}(\PGL(4))\dispunct; \qquad H^p(\PP^3) \tensor H^q(U_p) \implies H^{p+q}(U) \]
and, since the last bundle is trivial,
\[\dSum_{p+q = d} H^p(\PP^3) \tensor H^q(M) \cong H^d(M \times \PP^3) \dispunct.\]
Alternatively, all the differentials in the spectral sequence for the third bundle are $0$.
Since the Serre spectral sequence is natural, this will help us compute the differentials in the case of $U$.
We will also use our knowledge of the differentials in the $\PGL(4)$ case.

By \cref{base-cohomology}, $H^*(M)$ is isomorphic to $H^*(\PGL(4))$ via the orbit map.
This implies that the map $G_p \to M$ induces isomorphisms:
\[H^3(M) \cong H^3(G_p) \dispunct; \qquad H^5(M) \cong H^5(G_p) \dispunct.\]
In particular, the map
\[\QQ \cong H^3(M) \to H^3(U_p) \cong \QQ\]
is an isomorphism and
\[\QQ \cong H^5(M) \to H^5(U_p) \cong \QQ^2\] is injective.
Now, by the Leibniz rule for differentials in the Serre spectral sequence and the description of $H^*(U_p)$ from \cref{fiber-cohomology-ring}, it is enough to find the ranks of the differentials (see \cref{serre-SS})
\[d_4 \from H^3(U_p) \to H^4(\PP^3) \qquad \text{and} \qquad d_6 \from H^5(U_p) \to H^6(\PP^3) \dispunct.\]
By the isomorphism $H^3(M) \cong H^3(U_p)$, the differential $d_4$ vanishes but the injection $H^5(M) \to H^5(U_p)$ is not enough to determine if the differential $d_6$ has rank $1$ or $0$ (although the image of $H^5(M)$ must be in the kernel of $d_6$).

\begin{figure}
\centering
\begin{subfigure}[b]{0.4\linewidth}
\centering
\includegraphics{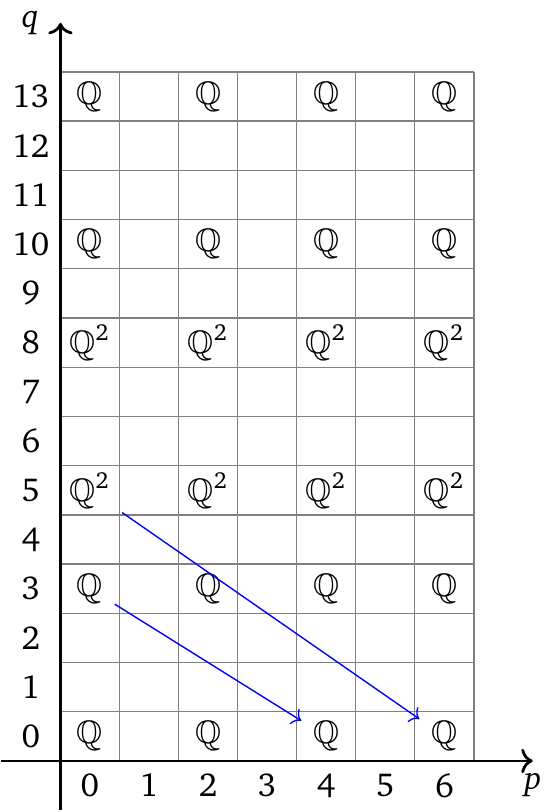}
\caption{pages $E_2 = \dots = E_6$}
\end{subfigure}
\begin{subfigure}[b]{0.4\linewidth}
\centering
\includegraphics{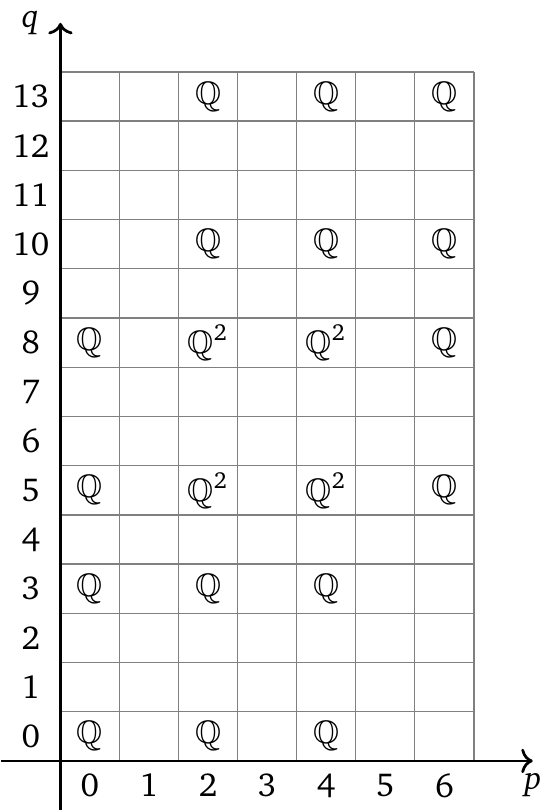}
\caption{pages $E_7 = \dots = E_\infty$}
\end{subfigure}
\caption{Serre spectral sequence for the bundle $U \to \PP^3$.
Only potentially non-zero differentials on \emph{generating degrees} are shown.}
\label{serre-SS}
\end{figure}

Since we are considering field coefficients, $H^*(U)$ is isomorphic to the associated graded as a vector space.
Thus, the Poincaré polynomial of $U$ is either
\[(1+t^3)(1+t^5)(1+t^7)(1+t^2 + t^4)\]
or
\[(1+t^3)(1+t^5)^2(1+t^2 + t^4 + t^6) \dispunct.\]
But the map $H^*(U) \to H^*(\PGL(4))$ is surjective (since the map $H^*(M) \to H^*(\PGL(4))$ is), so by yet another application of \cite[Theorem 2]{PS03}, the Poincaré polynomial of $U$ must be divisible by that of $\PGL(4)$, in particular by $(1+t^7)$.
This implies that the rank of $d^6$ is $1$ and that $H^6(\PP^3)$ is in the kernel of the pullback map for the fiber bundle $U \to \PP^3$.

Finally, to establish the ring structure, it is enough to note that in addition to the version of Leray--Hirsch from \cite{PS03}, the generators in degrees $3$, $5$ and $7$ cannot have any relations except those forced by graded commutativity since this is true for their images in $H^*(\PGL(4))$.
\end{proof}

\section{Rational cohomology of \texorpdfstring{$X_p$}{X(p)}}\label{case-work}

\subsection{Definitions and plan of attack}\label{plan}

We will suppress constant rational coefficients throughout this section and use $\widebar H$ to denote Borel--Moore homology.
Recall that for an orientable but not necessarily compact $2n$-manifold $M$, Poincaré duality takes the form
\[\widebar H_i(M) \cong H^{2n-i}(M) \cong (H_{2n-i}(M))^\vee \cong (H^{i}_c(M))^\vee \dispunct.\]

Recall that $X_p \subset \Pi_p \cong \CC^{19}$ and set $\Sigma_p = \Pi_p \setminus X_p = \Pi_p \cap \Sigma$, the set of \emph{singular} cubic polynomials that vanish at the point $p$.
Then by Alexander duality,
\begin{equation}\label{alexander-duality}
\widetilde H^i(X_p) = \widebar H_{37 - i}(\Sigma_p) \dispunct.
\end{equation}

\begin{remark}\label{andreotti--frankel}
The `discriminant locus' $\Sigma_p$ is a conical hypersurface in $\Pi_p$, being the vanishing locus of $\Delta_p = \Delta|_{\Pi_p}$.
The complex variety $X_p$, being the complement of a hypersurface, is affine and hence a $19$-dimensional Stein manifold.
Thus by the Andreotti--Frankel theorem, $H^i(X_p) = 0$ for $i > 19$.
Hence, by \cref{alexander-duality}, $\widebar H_{i}(\Sigma_p)$ can only be non-zero for $18 \le i \le 37$.
\end{remark}

Let $F \in \Sigma_p$ be a singular cubic polynomial and let $K$ be its singular locus.
Then $K$, as a subset of $\PP^3$, can be one of the following $11$ \emph{types} (see~\cite[Proposition 8]{Vassiliev99}):
\begin{singularity}
\item a point \label{point}
\item two distinct points \label{two-points}
\item a line
\item three points, not on a line \label{three-points}
\item a smooth conic contained in a plane $\PP^2 \subset \PP^3$
\item a pair of intersecting lines
\item four points, not on a plane \label{four-points}
\item a plane
\item three lines through a point, not all on the same plane
\item a smooth conic contained in a plane along with another point not on that plane
\item all of $\PP^3$ \label{everything}
\end{singularity}

These can be further classified into \emph{subtypes} depending on their configuration with respect to the marked point $p$.
This will not be relevant for most of the types; we list those that are relevant.
The names $P$, $Q$ etc.\ below for the points are for convenience, the sets of points are a priori \emph{unordered}: $\{P,Q\} = \{Q,P\}$ and so on.
\begin{singularity}
\item[(\ref{point})] a point $P$
\begin{singularity}[ref=\ref*{point}\alph*]
\item $P = p$\label{one-on}
\item $P \ne p$\label{one-off}
\end{singularity}
\item[(\ref{two-points})] two points $P$, $Q$
\begin{singularity}[ref=\ref*{two-points}\alph*]
\item $P = p$\label{two-on}
\item $P$ and $Q$ collinear with $p$, $P, Q \ne p$\label{two-collinear}
\item $P$ and $Q$ not collinear with $p$ \label{two-not-collinear}
\end{singularity}
\item[(\ref{three-points})] three points $P$, $Q$, $R$, not collinear
\begin{singularity}[ref=\ref*{three-points}\alph*]
\item $P=p$ \label{three-on}
\item $P$ and $Q$ collinear with $p$, $P, Q \ne p$\label{three-collinear}
\item $P$, $Q$ and $R$ coplanar with $p$, no two collinear with $p$\label{three-coplanar}
\item $P$, $Q$ and $R$ not coplanar with $p$\label{three-not-coplanar}
\end{singularity}
\item[(\ref{four-points})] four points $P$, $Q$, $R$, $S$, not coplanar
\begin{singularity}[ref=\ref*{four-points}\alph*]
\item $P=p$\label{four-on}
\item $P$ and $Q$ collinear with $p$, $P, Q \ne p$\label{four-collinear}
\item $P$, $Q$ and $R$ coplanar with $p$, no two collinear with $p$\label{four-coplanar}
\item no three coplanar with $p$\label{four-no-three-coplanar}
\end{singularity}
\end{singularity}

\begin{remark}
The \emph{types} correspond to singular loci that are equivalent under the $\PGL(4)$-action on $\PP^3$ and the \emph{subtypes} correspond to equivalence under the action of the subgroup $\Stab(p) \subset \PGL(4)$.
However, this will not be explicitly important for us.
\end{remark}

\begin{definition}\label{singular-polynomial}
For a singular locus $K$, denote by $L(K)$ the set of all polynomials in $\Sigma_p$ that are singular on all of $K$ (and perhaps elsewhere as well).
This is a vector space for any $K \subset \PP^3$.
\end{definition}

\begin{remark}
The subtypes are (partially) ordered by degeneracy: $i \le j$ if polynomials with singularity of subtype~$i$ can degenerate to a polynomial with singularity of subtype~$j$.
In the following we need to choose a rank function (i.e., monotonic integer-valued map) on this poset, we use
\[\deg(i) = 16 - \dim L(K)\]
for any $K$ of subtype $i$.
\end{remark}

\begin{definition}
For a manifold $M$ and natural number $n$, the \emph{ordered configuration space} of $n$ points on $M$ is given by
\[\PConf_n(M) \vcentcolon= \set*{(a_1,\dots,a_n) \in M^n}{a_i \ne a_j \text{ for } i \ne j} \dispunct.\]
This space comes with a natural action of the symmetric group $\fS_n$ by permuting the coordinates and the quotient is the \emph{unordered configuration space} $\UConf_n(M)$ of $n$ points on $M$.
\end{definition}

\begin{definition}
For any $A \subseteq \UConf_n(M)$, the \emph{sign local coefficients} on $A$, denoted by $\pm \QQ$, is given by the composition
\[\pi_1(A) \to \pi_1(\UConf_n(M)) \to \fS_n \to \{\pm 1\} \subset \QQ^\times\]
thought of as a representation on $\QQ$.
Explicitly, a loop in $A$ acts on $\QQ$ by the sign of the induced permutation on the $n$ points.
\end{definition}

The method of simplicial resolution ultimately produces for us a spectral sequence
\[E_{*,*} \implies \widebar H_{*}(\Sigma_p)\]
with the $E^1$ page described below.
For slightly more details see an entirely analogous description in \cite{Das18}; for proofs and constructions, see \cite{Vassiliev99}.

Let the index $i$ vary over all the \emph{subtypes} (not just the ones listed, but all of them).
Define
\[A_i \vcentcolon= \{\text{singular sets $K$ of subtype~$i$}\} \dispunct.\]

\begin{example}\label{parameter-space-example}
For the subtype~\ref{two-on}, a point $P = p$ and a point $Q \ne p$, we have $A_{\type{two-on}} = \{p\} \times (\PP^3 \setminus \{p\})$.
For the subtype~\ref{two-not-collinear}, two points not collinear with $p$, the space $A_{\type{two-not-collinear}}$ is an open set in $\UConf_2(\PP^3 \setminus \{p\})$.
\end{example}

There are spaces $F_i$ so that the $E^1$ page is given by
\begin{equation}\label{ss-terms}
E^1_{p,q} = \dSum_{\deg(i) = p} \widebar H_{p+q}(F_i) \dispunct.
\end{equation}
There are further spaces $\Phi_i$ and $\Lambda(K)$ as well as fiber bundles:
\[
  \begin{tikzcd}[column sep=tiny]
    L(K) \ar[hookrightarrow]{rr} & & F_i \dar[->>] \\
    \Lambda(K) \ar[hookrightarrow]{rr} & &\Phi_i \dar[->>] \\
 & K \rar[draw=none][description]{\in} & A_i
  \end{tikzcd}
\]
So to compute $\widebar H_*(F_i)$, we can use the Thom isomorphism
\begin{equation}\label{thom-isomorphism}
\widebar H_*(F_i) = \widebar H_{* -2 \dim_\CC L(K)} (\Phi_i) \dispunct.
\end{equation}

Unless $i$ is a subtype of \ref{point}, \ref{two-points}, \ref{three-points}, \ref{four-points} or \ref{everything}, $\widebar H_*(\Lambda(K)) = 0$ (\cite[proof of Proposition 9]{Vassiliev99}) and hence $\widebar H_* (\Phi_i) = 0$.
Now suppose $i$ is a subtype of \ref{point}, \ref{two-points}, \ref{three-points} or \ref{four-points}, i.e.\ $K \in A_i$ is a finite set of say $n$ points.
Then $A_i$ is a subset of $\UConf_n(\PP^3)$ and
\begin{equation}\label{local-coefficient-isomorphism}
\widebar H_*(\Phi_i) = \widebar H_{* - n + 1}(A_i; \pm \QQ) = H^{2\dim_\CC A_i + n - 1 - *}(A_i; \pm \QQ)   \dispunct.
\end{equation}

For the type \ref{everything} (note that \ref{everything} has only one subtype, itself), $A_{\type{everything}}$ is singleton, the only element being $K = \PP^3$.
The only polynomial singular on $K$ is $0$, so $L(K) = \{0\}$.
Thus $F_{\type{everything}} = \Phi_{\type{everything}} =  \Lambda(\PP^3)$.
Further, the space $\Phi_{\type{everything}} = \Lambda(\PP^3)$ is the open cone $\mathring{C} Z$ over
\[Z = \bigcup_{j \ne \type{everything}} \Phi_j\]
for certain gluings.
So we get a spectral sequence $e^r_{p,q} \implies H_{p+q} (Z)$ with
\[e^1_{p,q} = \dSum_{\substack{\deg(j) = p,\\ j \ne \type{everything}}} \widebar H_{p+q}(\Phi_j) \dispunct.\]
But then we also have
\[\widebar H_*(\mathring CZ) = H_*(CZ , Z) = \widetilde H_{*-1}(Z) \dispunct.\]

So the computation eventually reduces to computing $H^*(A_i; \pm \QQ)$ for the various subtypes of \ref{point}, \ref{two-points}, \ref{three-points} and \ref{four-points} (see~\cref{vanishing-cohomology,shape-of-page-1}) followed by bookkeeping and relatively standard arguments involving spectral sequences following \cite{Vassiliev99} (see~\cref{cohomology-of-X-p}).

\begin{remark}
We could keep track of the mixed Hodge structures throughout the entire computation, as in \cite{Tommasi05,Tommasi14} (see~also~\cite{Gorinov05}), but this ends up being unnecessary for our purposes.
\end{remark}

\subsection{Case work}

This section contains the details of the arguments to compute the various $H^*(A_i; \pm \QQ)$.
The main idea is decomposing these spaces as fiber bundles, where both the fiber and base are simpler.
In many instances the bases are $A_j$ for some lower $j$ and the computation is `inductive' or recursive.

First, a couple of general facts that we will use freely in the computation below:
\begin{lemma}[{\cite[Lemma 3.11]{Das18}}]\label{subspace-complement}
Let $H \cong \PP^k$ be a $k$-dimensional linear subspace of $\PP^n$ for some $0 \le k \le n$ and let $H^\perp$ be the (projectivized) orthogonal complement of $H$.
Then $\PP^n \setminus H$ deformation retracts to $H^\perp \cong \PP^{n-k-1}$.
\end{lemma}

\begin{lemma}\label{totaro-computations}
\[H^*(\UConf_2(\CC); \pm \QQ)
= 0 \dispunct.\]
\[H^*(\UConf_2(\PP^2); \pm \QQ) \cong \begin{cases*}
\QQ & if $* = 2, 4, 6$;\\
0 & otherwise.\end{cases*}\]
\end{lemma}
\begin{proof}
For $\UConf_2(\CC)$, we can use that $\PConf_2(\RR^{2n}) \simeq S^{2n-1}$ and the $\fS_2$ action is by the antipodal map, which is degree $1$.
Hence, by transfer, $H^*(\UConf_2(\RR^{2n}); \pm \QQ) = 0$.
For $H^*(\UConf_2(\PP^2); \pm \QQ)$ see \cite[Lemma~2B]{Vassiliev99}.
\end{proof}

Now we establish the cases where $H^*(A_i; \pm \QQ) = 0$, the recursive nature of the argument makes some of the cases relatively easy.
The remaining cases are treated in \cref{shape-of-page-1}.

\begin{proposition}\label{vanishing-cohomology}
If $i$ is \ref{two-collinear}, \ref{three-collinear}, \ref{three-coplanar}, \ref{four-collinear}, \ref{four-coplanar} or \ref{four-no-three-coplanar} then $H^*(A_i; \pm \QQ) = 0$.
\end{proposition}

\begin{proof}
Let's deal with each case in turn.

\begin{description}[wide,itemindent=*]

\item[\ref{two-collinear}. $P, Q \ne p$, but $P$, $Q$ and $p$ collinear]
Mapping $\{P,Q\} \mapsto L = \gen*{P,Q,p}$, the projective span of $P$, $Q$ and $p$, i.e.\ the line containing $P$, $Q$ and $p$, we get a map from $A_{\type{two-collinear}}$ to the space of lines in $\PP^3$ containing $p$, which is a $\PP^2 \subset G(1,3)$.
This is a fiber bundle
\[\begin{tikzcd}
\UConf_2(L \setminus p) \arrow[r,hook] & A_{\type{two-collinear}} \arrow[d,->>] \\ & \PP^2
\end{tikzcd}\]
and the local coefficients $\pm \QQ$ restricts to the fiber to the sign local coefficient on $\UConf_2(L \setminus p) \cong \UConf_2(\CC)$.
But $H^*(\UConf_2(\CC),\pm \QQ) = 0$ from \cref{totaro-computations}, so we are done.

\item[\ref{three-collinear}. $P, Q, R \ne p$, $P$, $Q$ and $p$ collinear, but $R$ not on that line]
Here, even though $P$, $Q$ and $R$ are a priori unordered, we can't (continuously) interchange $R$ with one of $P$ and $Q$.
So there is a well-defined map $\{P,Q,R\} \mapsto \{Q,R\}$ and we get a fiber bundle:
\[\begin{tikzcd}
\PP^3 \setminus \PP^1 \cong \PP^3 \setminus \gen*{P,Q,p} \rar[hookrightarrow] & A_{\type{three-collinear}} \dar[->>] \\ & A_{\type{two-collinear}}
\end{tikzcd}\]
The local coefficients $\pm \QQ$ on the total space pulls back from $\pm \QQ$ on base (that is, the map $\pi_1(A_{\type{three-collinear}}) \to \{\pm 1\}$ factors through $\pi_1(A_{\type{two-collinear}})$).
But as we just showed, $H^*(A_{\type{two-collinear}}; \pm \QQ) = 0$, so we are done.

\item[\ref{three-coplanar}. $P, Q, R \ne p$, coplanar with $p$ and no three of $P$, $Q$, $R$ and $p$ collinear]
Mapping
\[\{P,Q,R\} \mapsto H = \gen*{P,Q,R,p}\dispunct,\]
we get a fiber bundle:
\[\begin{tikzcd}
F \rar[hookrightarrow] & A_{\type{three-coplanar}} \dar[->>] \\ & \PP^2
\end{tikzcd}\]
The fiber is the space of three (unordered) non-collinear points on $H \setminus \{p\}$ and the local coefficients $\pm \QQ$ restricts to the local coefficients $\pm \QQ$ on $F \subset \UConf_3(\PP^2)$.
Since $\pi_1(F) \to \{\pm 1\}$ factors through $\fS_3$, we can go to the associated $\fS_3$ cover $\widetilde F \subset \PConf_3(\PP^2)$ and then, by transfer, $H^*(F; \pm \QQ)$ is the summand of $H^*(\widetilde F; \QQ)$ where $\fS_3$ acts by the sign representation.

But $\widetilde F$ can be identified with the fiber of $(P,Q,R,S) \mapsto S$, where $(P,Q,R,S)$ varies over all tuples in $\PConf_4(\PP^2)$ so that no three are collinear.
But $\PGL(3)$ acts freely and transitively on this open subset of $\PConf_4(\PP^2)$ and hence we have a fiber bundle:
\[\begin{tikzcd}
\widetilde F \rar[hookrightarrow] & \PGL(3) \dar[->>] \\ & \PP^2 \end{tikzcd}\]
The action of $\fS_3$ extends to $\PConf_4(\PP^2)$, permuting the first three points, so the action on the base is trivial.
The action on the total space extends to the action of the entire (connected) group $\PGL(3)$ by right multiplication, so is trivial on homology.
As a result, the $\fS_3$ action on $H^*(\widetilde F; \QQ)$ is trivial, which implies $H^*(F; \pm \QQ) = 0$, as needed.

\item[\ref{four-collinear}. $P$, $Q$, $R$ and $S$ not coplanar, $P$ and $Q$ collinear with $p$]
Note that if the four points are not coplanar, at most one pair can be collinear with $p$, so this determines the subset $\{P,Q\} \subset \{P,Q,R,S\}$.
The line $L = \gen*{P, Q, p}$ can be any line through $p$ that is not on the plane $\gen*{R, S, p}$ and fixing $L$, $P$ and $Q$ vary exactly in $\UConf_2(L \setminus \{p\}) \cong \UConf_2(\CC)$.
Thus mapping $\{P,Q,R,S\} \mapsto (L, \{R,S\})$ we get a fiber bundle:
\[\begin{tikzcd}
\UConf_2(L \setminus \{p\}) \arrow[r,hook] & A_{\type{four-collinear}} \dar[->>] \\
& \{(L, \{R,S\})\}
\end{tikzcd}\]
But again $H^*(\UConf_2(\CC),\pm \QQ) = 0$ from \cref{totaro-computations}, so we are done.

\item[\ref{four-coplanar}. $P$, $Q$ and $R$ coplanar with $p$, $S$ not on that plane and no two collinear with $p$]
Mapping
\[\{P,Q,R,S\} \mapsto \{P,Q,R\}\] we get a fiber bundle:
\[\begin{tikzcd}
\CC^3 \cong \PP^3 \setminus \gen*{P,Q,R} \rar[hookrightarrow] & A_{\type{four-coplanar}} \dar[->>] \\
& A_{\type{three-coplanar}}
\end{tikzcd}\]
Since $H^*(A_{\type{three-coplanar}},\pm \QQ) = 0$ by previous arguments, we are done.

\item[\ref{four-no-three-coplanar}. $P, Q, R, S \ne p$, no three coplanar with $p$]
By an argument analogous to the case of \ref{three-coplanar}, $A_{\type{four-no-three-coplanar}}$ has an $\fS_4$ cover by ordering the four points.
This cover is the fiber of the bundle $\PGL(3) \to \PP^3$, where $\PGL(3)$ is identified with five (ordered) points in $\PP^3$, no four of which are coplanar, by its free and transitive action.
The action of $\fS_4$ is again trivial on the base and on $H^*(\PGL(3))$, so we are done. \qedhere
\end{description}
\end{proof}

Recall that we have spectral sequences $E^r_{p,q} \implies \widebar H_{p+q} (\sigma)$ and $e^r_{p,q}$ that let us compute $\widebar H_*(F_{\type{everything}}) = \widetilde H_{*-1}(Z)$, where
\[Z = \bigcup_{i \ne \type{everything}} \Phi_i \dispunct.\]

\begin{proposition}\label{shape-of-page-1}
The spectral sequence $E^r_{p,q} \implies \widebar H_{p+q} (\sigma)$ has the page $E^1_{p,q}$ as in \cref{big-SS}.
The spectral sequence $e^r_{p,q} \implies H_{p+q} (Z)$ has the page $e^1_{p,q}$ as in \cref{small-SS}.
\end{proposition}

\begin{figure}
\centering
\includegraphics{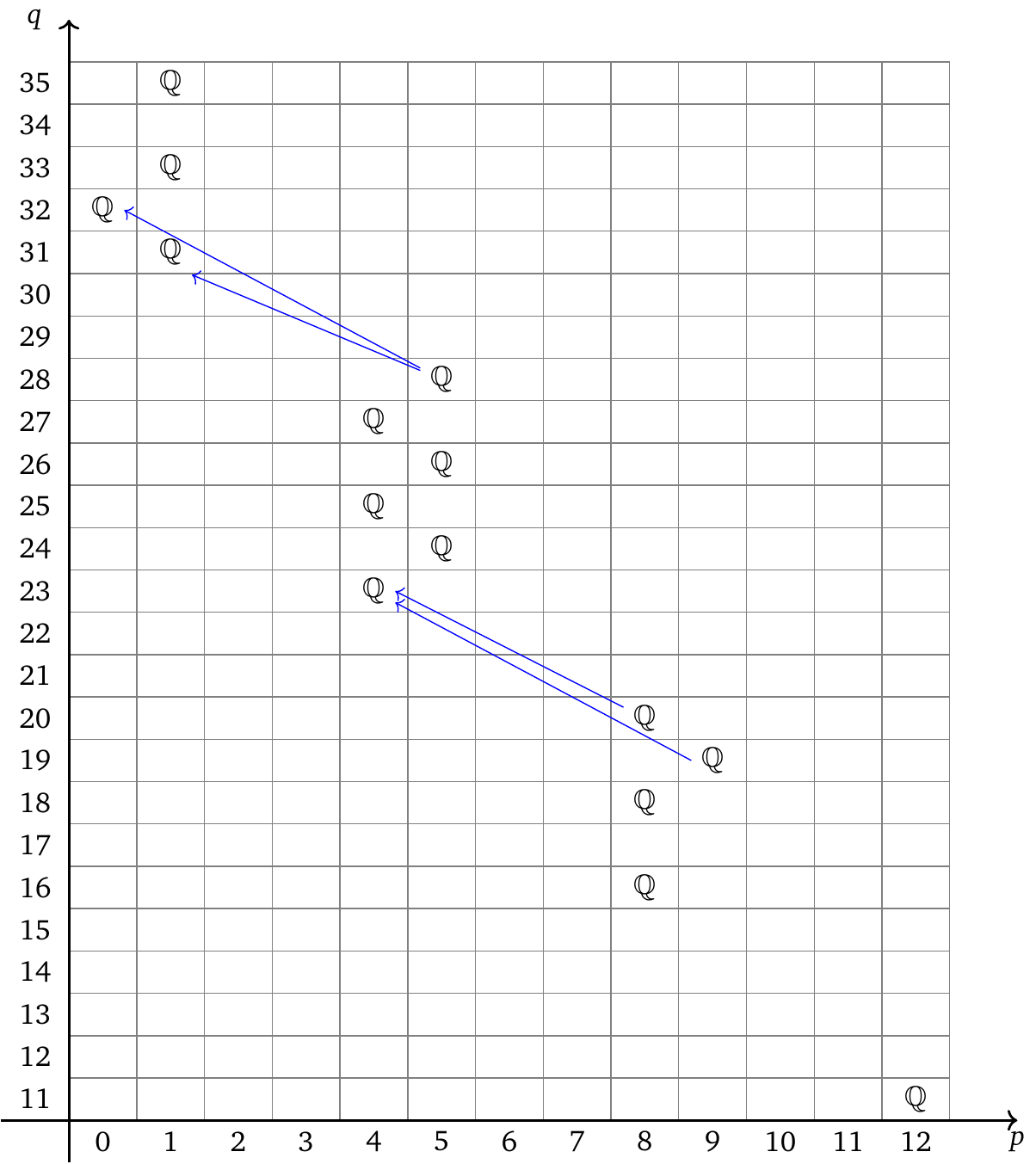}
\caption{Spectral sequence page $E^1_{p,q}$ for $\widebar H_{p+q} (\sigma)$ (with $0$s omitted) and all potentially non-zero differentials in subsequent pages}
\label{big-SS}
\end{figure}

\begin{figure}
\centering
\includegraphics{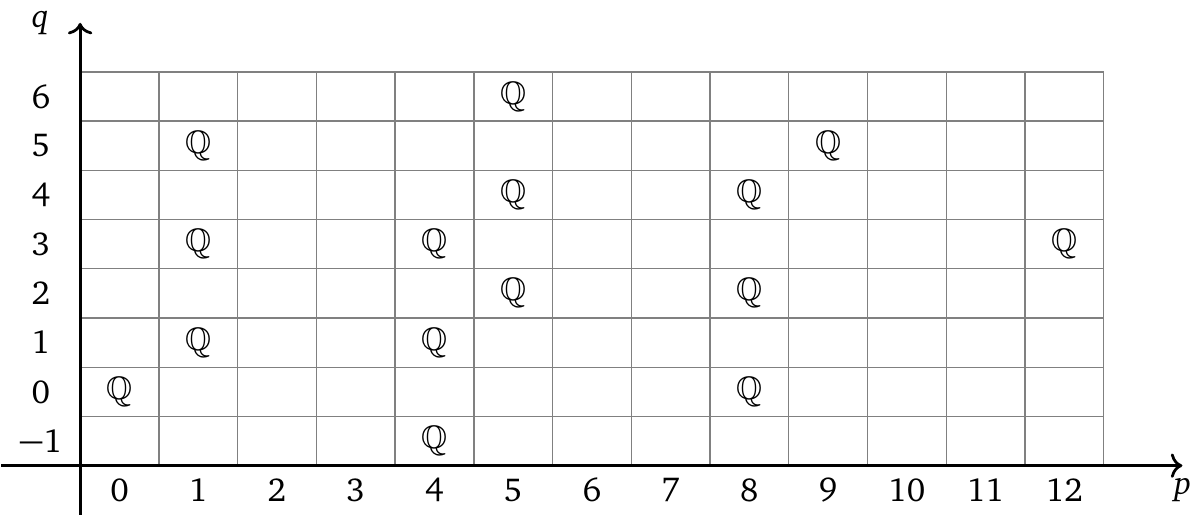}
\caption{Spectral sequence page $e^1_{p,q}$ for $H_{p+q}(Z)$ (with $0$s omitted)}
\label{small-SS}
\end{figure}

\begin{proof}
Recall that by construction, the terms of $E^1$ and $e^1$ are related by Thom isomorphisms
\[E^1_{p,q + 2(16-p)} \cong e^1_{p,q} \dispunct,\]
except for $p = 16$, where $e^1_{16,*} \equiv 0$.
So we first go through more case work to establish columns $p \ne 16$.

By \cref{ss-terms,thom-isomorphism,local-coefficient-isomorphism} and careful bookkeeping, it is enough to find $H^*(A_i; \pm \QQ)$ along with the numbers $\dim(A_i) = \dim_\CC(A_i)$ and $\dim(L(K)) = \dim_\CC(L(K))$ for $K \in A_i$, for the subtypes $i$ of \ref{point}, \ref{two-points}, \ref{three-points} and \ref{four-points} (see~\cref{numbers-table} for the relevant numerics).
Further, there are only seven subtypes remaining --- the ones not covered in \cref{vanishing-cohomology}.

\begin{table}
\centering
\begin{tabular}{crrrrrrrr}
\toprule
$i$ & \ref{one-on} & \ref{one-off} & \ref{two-on} & \ref{two-not-collinear} & \ref{three-on} & \ref{three-not-coplanar} & \ref{four-on} & \ref{everything}\\
\midrule
$\dim A_i$ & $0$ & $3$ & $3$ & $6$ & $6$ & $9$ & $9$ & $0$ \\
$\dim L(K)$ & $16$ & $15$ & $12$ & $11$ & $8$ & $7$ & $4$ & $0$ \\
\bottomrule
\end{tabular}
\caption{$\dim A_i$ and $\dim L(K)$ for $K \in A_i$ for each subtype $i$ excepted in \cref{vanishing-cohomology}.} \label{numbers-table}
\end{table}

\begin{description}[wide]
\item[\ref{one-on}. $P = p$]
$A_{\type{one-on}} = \{p\}$ and the coefficients $\pm \QQ$ are trivial, so
\[H^*(A_{\type{one-on}}; \pm \QQ) = H^*(\{p\}) = \begin{cases*} \QQ & if $* = 0$;\\ 0 & otherwise.\end{cases*}\]
This contributes to $E^1_{0,32} \cong e^1_{0,0}$ since $\dim(A_{\type{one-on}}) =  0$ and $\dim(L(K)) = 16$.

\item[\ref{one-off}. $P \ne p$]
$A_{\type{one-off}} = \PP^3 - p \simeq \PP^2$.
Again, the coefficients are trivial since there is only one point, so
\[H^*(A_{\type{one-off}}; \pm \QQ) = H^*(\PP^2) = \begin{cases*} \QQ & if $* = 0, 2, 4$;\\ 0 & otherwise.\end{cases*}\]
This contributes to $E^1_{1,31} \cong e^1_{1,1}$, $E^1_{1,33} \cong e^1_{1,3}$ and $E^1_{1,35} \cong e^1_{1,5}$ since $\dim(A_{\type{one-off}}) =  3$ and $\dim(L(K)) = 15$.

\item[\ref{two-on}. $P=p$, $Q \ne p$]
$A_{\type{two-on}} = \{p\} \times \PP^3 \setminus \{p\} \cong (\PP^2)$ and the coefficients are trivial.
So,
\[H^*(A_{\type{two-on}}; \pm \QQ) = \begin{cases*} \QQ & if $* = 0, 2, 4$;\\ 0 & otherwise.\end{cases*}\]
This contributes to $E^1_{4,23} \cong e^1_{4,-1}$, $E^1_{4,25} \cong e^1_{4,1}$ and $E^1_{4,27} \cong e^1_{4,3}$ since $\dim(A_{\type{two-on}}) =  3$ and $\dim(L(K)) = 12$.

\item[\ref{two-not-collinear}. $P$, $Q$ and $p$ not collinear]
The three points not being collinear is equivalent to the lines $\gen*{P,p}$ and $\gen*{Q,p}$ being distinct (lines through $p$). Hence mapping $\{P,Q\} \mapsto \{\gen*{P,p},\gen*{Q,p}\}$, we get a fiber bundle
\[\begin{tikzcd}
 \CC^2 \cong (\gen*{P,p} \setminus p)\times(\gen*{Q,p} \setminus p)\rar[hookrightarrow] & A_{\type{two-not-collinear}} \dar[->>] \\
 & \UConf_2(\PP^2_p)
\end{tikzcd}\]
where $\PP^2_p \cong \PP^2$ is the space of lines through $p$.
The coefficients on the total space pull back from $\pm \QQ$ coefficients on the base, hence by \cref{totaro-computations},
\[H^*(A_{\type{two-not-collinear}}; \pm \QQ) \cong H^*(\UConf_2(\PP^2); \pm \QQ) = \begin{cases*} \QQ & if $* = 2, 4, 6$;\\ 0 & otherwise.\end{cases*}\]
This contributes to $E^1_{5,24} \cong e^1_{5,2}$, $E^1_{5,26} \cong e^1_{5,4}$ and $E^1_{5,28} \cong e^1_{5,6}$ since $\dim(A_{\type{two-not-collinear}}) = 6$ and $\dim(L(K)) = 11$.

\item[\ref{three-on}. $P = p$, $Q$ and $R$ not coplanar with $p$]
$A_{\type{three-on}} = \{p\} \times A_{\type{two-not-collinear}}$ and the coefficients pull back from the $\pm \QQ$ coefficients on $A_{\type{two-not-collinear}}$.
Hence,
\[H^*(A_{\type{three-on}}; \pm \QQ) = \begin{cases*} \QQ & if $* = 2, 4, 6$;\\ 0 & otherwise.\end{cases*}\]
This contributes to $E^1_{8,16} \cong e^1_{8,0}$, $E^1_{8,18} \cong e^1_{8,2}$ and $E^1_{8,20} \cong e^1_{8,4}$ since $\dim(A_{\type{two-on}}) =  6$ and $\dim(L(K)) = 8$.

\item[\ref{three-not-coplanar}. $P$, $Q$ and $R$ not coplanar with $p$]
Mapping $\{P,Q,R\} \mapsto \gen*{P,Q,R}$, we get a fiber bundle whose base is $(\PP^3)^\vee \setminus p^\perp \cong \CC^3$ and the fiber is the space of non-collinear triples of points in $\PP^2$, whose cohomology with $\pm \QQ$ coefficients is the same as that of $\UConf_3(\PP^2)$, by \cite[Lemma 4]{Vassiliev99}.
Thus, using \cref{totaro-computations},
\[H^*(A_{\type{three-not-coplanar}}; \pm \QQ) = \begin{cases*} \QQ & if $* = 6$;\\ 0 & otherwise.\end{cases*}\]
This contributes to $E^1_{9,19} \cong e^1_{9,5}$ since $\dim(A_{\type{three-not-coplanar}}) = 9$ and $\dim(L(K)) = 7$.

\item[\ref{four-on}. $P = p$, $Q$, $R$ and $S$ not coplanar with $p$]
$A_{\type{three-on}} = \{p\} \times A_{\type{three-not-coplanar}}$ and the coefficients pull back from the $\pm \QQ$ coefficients on $A_{\type{three-not-coplanar}}$.
Hence
\[H^*(A_{\type{four-on}}; \pm \QQ) = \begin{cases*} \QQ & if $* = 6$;\\ 0 & otherwise.\end{cases*}\]
This contributes to $E^1_{12,11} \cong e^1_{12,3}$ since $\dim(A_{\type{four-on}}) = 9$ and $\dim(L(K)) = 4$.
\end{description}

Thus we've computed the pages $e^1_{p,q}$ and $E^1_{p,q}$ except the $p = 16$ column of the latter.
For \ref{everything}, $L(K) = 0$, so $E^1_{16,q} \cong \widebar H_{16+q}(\Phi_{\type{everything}})$.
Now, if any term with $1 \le d = p+q \le 16$ remains non-zero in $e^\infty_{p,q}$, then it would appear as $\widebar H_{d+1}(\Phi_{\type{everything}})$ and hence as a term $E^1_{16, d-15}$, which cannot interact with any of the other terms, by the shapes of the other columns, which we have already determined.
That means $0 \ne \widebar H_{d+1} (\sigma) \cong \widetilde H^{37-d} (X_p)$, which is a contradiction with $X_p$ being a $19$-dimensional Stein manifold, as in \cref{andreotti--frankel}.
This implies, given the shape of $e^1_{p,q}$, that $\widebar H_*(\Phi_{\type{everything}}) \equiv 0$, so we have also verified $E^1_{16,*}$.
\end{proof}

\begin{proposition}\label{cohomology-of-X-p}
The spectral sequence $E^r_{p,q}$ degenerates at $r = 1$ and hence the (rational) Poincaré polynomials of $X_p$ and $U_p$ are given by:
\begin{align*}
P(X_p; t) &= (1 + t)(1+t^3)(1+t^5)^2\\
P(U_p; t) &= (1+t^3)(1+t^5)^2
\end{align*}
\end{proposition}
\begin{proof}
Recall that $E^r_{p,q} \implies \widebar H_{p+q} (\sigma) \cong \widetilde H^{37 - p - q}(\sigma)$.
The page $E^1_{p,q}$ is quite sparse to begin with, the only potentially non-zero differentials (on any page) are shown in \cref{big-SS}.
By \cref{hypersurface-complement}, since $X_p = \Pi_l \setminus \cV(\Delta_p)$, we must have
\[P(X_p; t) = P(\CC^\times; t)P(U_p; t) = (1+t)P(U_p; t) \dispunct.\]
This shows that $H^4(X_p) \cong \widebar H_{33}(\sigma)$ and $H^{10}(X_p) \cong \widebar H_{27}(\sigma)$ cannot be $0$, which means all those differentials must vanish.
So $E^\infty_{p,q} \cong E^1_{p,q}$ and there are no extension problems with rational coefficients.
It is then straightforward to factor the polynomials in the given manner.
\end{proof}

\printbibliography

\end{document}